\documentclass{amsart}
\usepackage{latexsym,amssymb,amsthm,amsmath}

\usepackage{amssymb}
\usepackage{amsmath}

\theoremstyle{plain}
\newtheorem{theorem}{Theorem}[section]
\newtheorem*{Theorem B}{Theorem B}
\newtheorem*{Theorem A}{Theorem A}

\newtheorem{proposition}{Proposition}[section]
\newtheorem{corollary}{Corollary}[section]

\newtheorem{example}{Example}[section]
\numberwithin{equation}{section}

\theoremstyle{remark}

\begin{document}
\title[Invariant submanifolds of generalized Sasakian-space-forms]{Invariant submanifolds of generalized Sasakian-space-forms}
\author[S. K. Hui, S. Uddin, A.H. Alkhaldi and P. Mandal]{Shyamal Kumar Hui, Siraj Uddin, Ali H. Alkhaldi and Pradip Mandal}
\subjclass[2000]{53C15, 53C40}
\keywords{generalized Sasakian-space-forms, invariant submanifold, semi parallel submanifold, totally geodesic, semi-symmetric metric connection, Ricci soliton.}
\begin{abstract}
The present paper deals with the study of invariant submanifolds of generalized Sasakian-space-forms
with respect to Levi-Civita connection as well as semi-symmetric metric connection.
We provide some examples of such submanifolds and obtain many new results including, the necessary
and sufficient conditions under which the submanifolds are totally geodesic. The Ricci solitons of such submanifolds
are also studied.
\end{abstract}
\maketitle
\section{Introduction}
It is well known that in differential geometry the curvature of a Riemannian
manifold plays a basic role and the sectional curvatures of a manifold determine
the curvature tensor $\overline{R}$ completely. A Riemannian manifold with constant sectional
curvature $c$ is called a real-space-form and its curvature tensor $R$ satisfies
the condition
\begin{align}
\label{eqn1.1}
\overline{R}(X,Y)Z = c\{g(Y,Z)X-g(X,Z)Y\}.
\end{align}
Models for these spaces are the Euclidean spaces ($c=0$), the spheres ($c>0$)
and the hyperbolic spaces $(c<0)$.

In contact metric geometry, a Sasakian manifold with constant $\phi$-sectional curvature
is called Sasakian-space-form and the curvature tensor of such a manifold is given by
\begin{align}
\label{eqn1.2}
\overline{R}(X,Y)Z &= \frac{c+3}{4}\big\{g(Y,Z)X-g(X,Z)Y\big\}\\
\nonumber&+\frac{c-1}{4}\big\{g(X,\phi Z)\phi Y
- g(Y,\phi Z)\phi X + 2g(X,\phi Y)\phi Z\big\}\\
\nonumber&+\frac{c-1}{4}\big\{\eta(X)\eta(Z)Y - \eta(Y)\eta(Z)X\\
\nonumber&+g(X,Z)\eta(Y)\xi - g(Y,Z)\eta(X)\xi\big\}.
\end{align}
These spaces can also be modeled depending on $c>-3$, $c=-3$ or $c<-3$.

As a generalization of Sasakian-space-form, in \cite{ALEGRE1}
Alegre, Blair and Carriazo
introduced and studied the notion of a generalized
Sasakian-space-form with the existence of such notions by several interesting examples.
An almost contact metric manifold
$M(\phi,\xi,\eta,g)$ is called generalized Sasakian-space-form
if there exist three functions
$f_1$, $f_2$, $f_3$ on $\overline{M}$ such that \cite{ALEGRE1}
\begin{align}
\label{eqn1.3}
\overline{R}(X,Y)Z &=f_1\big\{g(Y,Z)X-g(X,Z)Y\big\}\\
\nonumber&+f_2\big\{g(X,\phi Z)\phi Y - g(Y,\phi Z)\phi X + 2g(X,\phi Y)\phi Z\big\}\\
\nonumber&+f_3\big\{\eta(X)\eta(Z)Y - \eta(Y)\eta(Z)X\\
\nonumber&+g(X,Z)\eta(Y)\xi - g(Y,Z)\eta(X)\xi\big\}
\end{align}
for all vector fields $X$, $Y$, $Z$ on $\overline{M}$, where $\overline{R}$
is the curvature tensor of $\overline{M}$ and such a manifold of dimension
$(2n+1)$, $n>1$ (the condition $n>1$ is assumed throughout the paper),
is denoted by $\overline{M}^{2n+1}(f_1,f_2,f_3)$.

In particular, if $f_1 = \frac{c+3}{4}$, $f_2 = f_3 = \frac{c-1}{4}$
then the generalized Sasakian-space-forms reduces to the notion of
Sasakian-space-forms. But it is to be noted that generalized Sasakian-space-forms
are not merely generalization of Sasakian-space-forms. It also contains a large class of
almost contact manifolds. For example it is known that \cite{ALEGRE2} any three dimensional
$(\alpha, \beta)$-trans Sasakian manifold with $\alpha$, $\beta$ depending on $\xi$ is a
generalized Sasakian-space-form. However, we can find generalized Sasakian-space-forms with
non-constant functions and arbitrary dimensions.

The generalized Sasakian-space-forms have been studied by several authors such as
Alegre and Carriazo (\cite{ALEGRE2}, \cite{ALEGRE3}, \cite{ALEGRE4}), Belkhelfa et al. \cite{BEL}, Carriazo \cite{CARR}, 
Al-Ghefari et al. \cite{SHAHID}, Gherib et al. \cite{GHERIB}, Hui et al. (\cite{HUI2}, \cite{HUI2}), Kim \cite{KIM} and many others.

In modern analysis, the geometry of submanifolds has become a subject of growing interest for its significant
applications in applied mathematics and theoretical physics. For instance,  the notion of invariant submanifold
is used to discuss properties of non-linear autonomous system \cite{inv}. For totaly geodesic submanifolds, the geodesics of the
ambient manifolds remain geodesics in the submanifolds. Hence totaly geodesic submanifolds are also very much
important in physical sciences. The study of geometry of invariant submanifolds was initiated by Bejancu and
Papaghuic \cite{bejancu}. In general the geometry of an invariant submanifold inherits almost all properties of the
ambient manifold. The invariant submanifolds have been studied by many geometers to different extent such
as \cite{HKMA}, \cite{KEN}, \cite{KON}, \cite{ASMS1}, \cite{SHAIKH2}, \cite{ISHI}, \cite{YIL} and others.\\
%%%%%%%%%%%%%%%%%%%%%%%%%%%%%%%%%
\indent Motivated by the above studies the present paper deals with the study of invariant submanifolds
of generalized Sasakian-space-forms. The paper is organized as follows. Section $2$ is concerned with
some preliminaries formulas and definitions. In this section, we provide some non-trivial examples of invariant, anti-invariant and proper slant submanifolds. Section $3$ is devoted to the study of invariant submanifolds
of generalized Sasakian-space-forms. In this section we study parallel, semiparallel
and 2-semiparallel invariant submanifolds of generalized sasakian-space-forms. Section $4$ deals with the
study of  invariant submanifolds of generalized Sasakian-space-forms with respect to semi-symmetric metric connection.\\
\indent In $1982$, Hamilton \cite{HAM1} introduced the notion of Ricci flow to find the canonical metric on a smooth manifold. Then Ricci flow
has become a powerful tool for the study of Riemannian manifolds, especially for those manifolds with positive curvature.
Perelman \cite {PER1} used Ricci flow and its surgery to prove Poincare conjecture.
The Ricci flow is an evolution equation for metrics on a Riemannian manifold defined as follows:
\begin{equation*}
\frac{\partial}{\partial t} g_{ij}(t)=-2R_{ij}
\end{equation*}
\indent A Ricci soliton emerges as the limit of the solutions of the Ricci flow. A solution to the Ricci flow is called Ricci soliton if it moves
only by one parameter group of diffeomorphism and scaling. A Ricci solitons $(g,V,\lambda)$ on a Riemannian manifold $(M,g)$ is a generalization
of an Einstein metric such that \cite{HAM2}
\begin{equation}
\label{eqn1.1d}
\pounds_V g+2S+2\lambda g=0,
\end{equation}
where $S$ is Ricci tensor, $\pounds_V$ is the Lie derivative operator along the vector field $V$ on $M$ and $\lambda$ is a real number.
The Ricci soliton is to be shrinking, steady and expanding according as $\lambda$ is negative, zero and positive respectively.\\
\indent During the last two decades, the geometry of Ricci solitons has been the focus of attention of many mathematicians. In particular
it has become more important after Perelman applied Ricci solitons to solve the long standing Poincare conjecture posed in $1904$.
In \cite{SHAR} Sharma studied the Ricci solitons in contact geometry. Thereafter, Ricci solitons in contact metric manifolds have been studied
by various authors such as Bejan and Crasmareanu \cite{BEJAN}, Hui et al. (\cite{CHS}, \cite{SKH2}-\cite{SKH9}), Chen and Deshmukh \cite{CD}, Deshmukh et al. \cite{Detal}, He and Zhu \cite{HE}, Tripathi \cite{TRIP} and many others.\\
\indent Ricci soliton on invariant submanifold of a generalized-Sasakian-space form is studied in section $5$. Finally in the last section, we
obtain some equivalent conditions of such notions.
\section{Preliminaries}
In an almost contact metric manifold, we have \cite{BLAIR}
\begin{align}
\label{eqn2.1}
\phi^2(X) = -X+\eta(X)\xi, \phi \xi=0,
\end{align}
\begin{align}
\label{eqn2.2}
\eta(\xi) = 1, g(X,\xi) = \eta(X), \eta(\phi X) = 0,
\end{align}
\begin{align}
\label{eqn2.3}
g(\phi X,\phi Y) = g(X,Y)-\eta(X)\eta(Y),
\end{align}
\begin{align}
\label{eqn2.4} g(\phi X,Y) = -g(X,\phi Y),
\end{align}
\begin{align}
\label{eqn2.5}
(\overline{\nabla}_X\eta)(Y) = g(\overline{\nabla}_X\xi,Y).
\end{align}
From (\ref{eqn1.3}) we have in a generalized Sasakian-space-form
$\overline{M}^{2n+1}(f_1,f_2,f_3)$,
\begin{align}
\label{eqn2.6}
(\overline{\nabla}_X\phi)(Y) = (f_1-f_3)[g(X,Y)\xi - \eta(Y)X],
\end{align}
\begin{align}
\label{eqn2.7}
\overline{\nabla}_X\xi = -(f_1-f_3) \phi X,
\end{align}
\begin{align}
\label{eqn2.8}
\overline{Q}X = (2nf_1+3f_2-f_3)X-\{3f_2+(2n-1)f_3\}\eta(X)\xi,
\end{align}
\begin{align}
\label{eqn2.9}
\overline{S}(X,Y) = (2nf_1+3f_2-f_3)g(X,Y)-\{3f_2+(2n-1)f_3\}\eta(X)\eta(Y),
\end{align}
\begin{align}
\label{eqn2.10}
\overline{r} = 2n(2n+1)f_1 + 6nf_2 - 4nf_3,
\end{align}
\begin{align}
\label{eqn2.11}
\overline{R}(X,Y)\xi = (f_1 - f_3)\{\eta(Y)X - \eta(X)Y\},
\end{align}
\begin{align}
\label{eqn2.12}
\overline{R}(\xi,X)Y = (f_1 - f_3)\{g(X,Y)\xi - \eta(Y)X\},
\end{align}
\begin{align}
\label{eqn2.13}
\eta(\overline{R}(X,Y)Z) = (f_1 - f_3)\{g(Y,Z)\eta(X) - g(X,Z)\eta(Y)\},
\end{align}
\begin{align}
\label{eqn2.14}
\overline{S}(X,\xi) = 2n(f_1 - f_3)\eta(X),
\end{align}
\begin{align}
\label{eqn2.15}
\overline{S}(\xi,\xi) = 2n(f_1 - f_3)
\end{align}
for all $X$, $Y$, $Z$ on $\overline{M}^{2n+1}(f_{1},f_{2},f_{3})$ and $\overline{\nabla}$
denotes the Levi-Civita connection on $\overline{M}$ and $\overline{S}$ is the Ricci tensor
and $\overline{r}$ is the scalar curvature of $\overline{M}$.

Let $M$ be a submanifold of a generalized Sasakian-space-form
$\overline{M}^{2n+1}(f_1,f_2,f_3)$. Also, let $\nabla$ and $\nabla^\perp$ be the induced
connections on the tangent bundle $TM$ and the normal bundle $T^\perp{M}$ of $M$, respectively.
Then the Gauss and Weingarten formulae are given by
\begin{align}
\label{eqn2.16}
\overline{\nabla}_XY = \nabla_XY +h(X,Y)
\end{align}
and
\begin{align}
\label{eqn2.17}
\overline{\nabla}_XV = -A_VX + \nabla_X^{\perp}V
\end{align}
for all $X,Y\in\Gamma(TM)$ and $V\in\Gamma(T^{\perp}M)$, where $h$ and $A_V$ are second fundamental form and shape operator (corresponding to the normal vector field V), respectively for the immersion of $M$ into $\overline{M}$. The second fundamental form $h$ and the shape operator $A_V$ are related by \cite{YANOKON}
\begin{align}
\label{eqn2.18}
g(h(X,Y),V) = g(A_VX,Y)
\end{align}
for any $X,Y\in\Gamma(TM)$ and $V\in\Gamma(T^{\perp}M)$. If $h=0$, then the submanifold is said to be totally geodesic.
 Also for any smooth function $f$ on a manifold we have
\begin{align}
\label{eqn2.19}
h(fX,Y) = fh(X,Y).
\end{align}
For the second fundamental form, the first and second covariant derivatives of $h$ are defined by
\begin{align}
\label{eqn2.20}
(\overline{\nabla}_Xh)(Y,Z) = {\nabla}^{\perp}_X(h(Y,Z))- h({\nabla}_XY,Z)- h(Y,{\nabla}_XZ)
\end{align}
and
\begin{align}
 \label{eqn2.20a}
  (\overline{\nabla}^{2}h)(Z,W,X,Y)= &= (\overline{\nabla}_{X}\overline{\nabla}_{Y}h)(Z,W) \\
 \nonumber &=\nabla^{\bot}_{X}((\nabla_{Y}h)(Z,W))-(\overline{\nabla}_{Y}h)(\nabla_{X}Z,W)\\
 \nonumber&-(\overline{\nabla}_{X}h)(Z,\nabla_{Y}W)-(\overline{\nabla}_{\nabla_{X}Y}h)(Z,W)
\end{align}
for any vector fields $X$, $Y$, $Z$, $W$ tangent to $M$. Then $\overline{\nabla}h$ is a normal bundle valued tensor
of type $(0,\,3)$ and is called the third fundamental form of $M$, $\overline{\nabla}$ is called the
Vander-Waerden-Bortolotti connection
of $\overline{M}$, i.e. $\overline{\nabla}$ is the connection in $TM\oplus T^\perp M$ built with
 $\nabla$ and $\nabla^\perp$. If $\overline{\nabla}h = 0$,
then $M$ is said to have parallel second fundamental form or the submanifold $M$ is said to be parallel \cite{JD1}.
An immersion is said to be semiparallel if
\begin{align}
\label{eqn2.21}
\overline{R}(X,Y)\cdot{h} = (\overline{\nabla}_X{\overline{\nabla}_Y} - {\overline{\nabla}_Y}{\overline{\nabla}_X} - {\overline{\nabla}_{[X,Y]}})h = 0
\end{align}
holds for all vector fields $X$,$Y$ tangent to $M$ \cite{JD1}, where $\overline{R}$ denotes the
 curvature tensor of the connection $\overline{\nabla}$. Semiparallel immersion have also been studied in \cite{JD2}, \cite{DILLEN}.
In \cite{ARSLAN} Arslan et al. defined and studied submanifolds satisfies the condition
\begin{align}
\label{eqn2.22}
\overline{R}(X,Y)\cdot\overline{\nabla}{h} = 0
\end{align}
for all vector fields $X$, $Y$ tangent to $M$ and such submanifolds are called 2-semiparallel.
In this connection it may be mentioned that $\ddot{\mbox{O}}$zg$\ddot{\mbox{u}}$r and Murathan studied
semiparallel and 2-semiparallel invariant submanifolds of LP-Sasakian manifolds.
From (\ref{eqn2.21}), we get
\begin{align}
\label{eqn2.23}
&(\overline{R}(X,Y)\cdot{h})(Z,U)\\
\nonumber&=  R^\perp(X,Y)h(Z,U) - h(R(X,Y)Z,U) - h(Z,R(X,Y)U)
\end{align}
for all vector fields $X$, $Y$, $Z$ and $U$ where
\begin{align*}
R^\perp(X,Y) = [\nabla_{X}^\perp,\nabla_{Y}^\perp] - \nabla_{[X,Y]}^\perp
\end{align*}
and $\overline{R}$ denotes the curvature tensor of $\overline{\nabla}$. In a similar way, we can write
\begin{align}
\label{eqn2.24}
&(\overline{R}(X,Y)\cdot \overline{\nabla}h)(Z,U,W)\\
\nonumber&= R^\perp(X,Y)(\overline{\nabla}h)(Z,U,W) - (\overline{\nabla}h)(R(X,Y)Z,U,W)\\
\nonumber&-(\overline{\nabla}h)(Z,R(X,Y)U,W)-(\overline{\nabla}h)(Z,U,R(X,Y)W)
\end{align}
for all vector fields $X$, $Y$, $Z$, $U$ and $W$ tangent to $M$ and $(\overline{\nabla}h)(Z,U,W) = (\overline{\nabla}_{Z}h)(U,W)$ \cite{ARSLAN}
and $R$ is the curvature tensor of $M$.

A transformation of a $(2n+1)$-dimensional Riemannian manifold $\overline{M}$,
which transforms every geodesic circle of $\overline{M}$
into a geodesic circle, is called a concircular transformation \cite{YANO}.
The interesting invariant of a concircular transformation is the concircular
curvature tensor $\overline{C}$, which is defined by \cite{YANO}
\begin{align}
\label{eqn2.25}
\overline{C}(X,Y)Z = \overline{R}(X,Y)Z-\frac{\overline{r}}{2n(2n+1)}\big[g(Y,Z)X-g(X,Z)Y\big],
\end{align}
where $\overline{r}$ is the scalar curvature of the manifold.

By virtue of (\ref{eqn2.11}) and (\ref{eqn2.12}) for a generalized Sasakian-space-form $\overline{M}^{2n+1}(f_1,f_2,f_3)$, it follows from (\ref{eqn2.25}) that
\begin{align}
\label{eqn2.26}
\overline{C}(X,Y)\xi = \big[f_1-f_3-\frac{\overline{r}}{2n(2n+1)}\big]\big[\eta(Y)X-\eta(X)Y\big],
\end{align}
\begin{align}
\label{eqn2.27}
\overline{C}(\xi,X)Y =\big[f_1-f_3-\frac{\overline{r}}{2n(2n+1)}\big]\big[g(X,Y)\xi-\eta(Y)X\big].
\end{align}
Also, we have
\begin{align}
\label{eqn2.28}
\left(\overline{C}(X,Y)\cdot{h}\right)(Z,U)&=R^\perp(X,Y)h(Z,U)\\
\nonumber&-h\left(C(X,Y)Z,U\right)- h\left(Z,C(X,Y)U\right),
\end{align}
\begin{align}
\label{eqn2.29}
&\left(\overline{C}(X,Y)\cdot \overline{\nabla}h\right)(Z,U,W) \\
\nonumber&=R^\perp(X,Y)(\overline{\nabla}h)(Z,U,W)-(\overline{\nabla}h)\left(C(X,Y)Z,U,W\right)\\
\nonumber&-(\overline{\nabla}h)\left(Z,C(X,Y)U,W\right)-(\overline{\nabla}h)\left(Z,U,C(X,Y)W\right),
\end{align}
where $C(X,Y)Z$ is the concircular curvature tensor of $M$.

A submanifold $M$ of an almost contact metric manifold $\overline{M}^{2n+1}$ is said to be
totally umbilical if
\begin{align}
\label{eqn2.29a}
  h(X,Y)=g(X,Y)H
\end{align}
for any vectors fields $X,Y\in \Gamma(TM)$, where $H$ is the mean curvature of $M$. Moreover, if $h(X,Y)=0$ for all $X,Y \in \Gamma(TM)$
 then $M$ is said to be {\it{totally geodesic}} and if $H=0$ then $M$ is {\it{minimal}} in $\overline{M}.$

Analogous to almost Hermitian manifolds, the invariant and anti-invariant submanifolds are depend on the behaviour of almost contact metric structure $\phi$.

A submanifold $M$ of an almost contact metric manifold $\overline{M}$ is said to be {\it{invariant}} if
the structure vector field $\xi$ is tangent to $M$ at every point of $M$ and $\phi X$ is tangent to $M$ for any
vector field $X$ tangent to $M$ at every point of $M$, that is $\phi(TM)\subset TM$ at every point of $M$.

On the other hand, $M$ is said to be {\it{anti-invariant}} if for any $X$ tangent to $M$, then $\phi X$ in normal to $M$, i.e., $\phi(TM)\subset T^\perp M$ at every point of $M$, where $T^\perp M$ is the normal bundle of $M$.

There is another class of submanifolds of almost Hermitian manifolds, called slant submanifolds 
introduced by B. Y. Chen \cite{Chen2}. Later, J.L. Cabrerizo et al. \cite{Cab1} defined and studied 
slant submanifolds of almost contact metric manifolds.

For each non-zero vector $X$ tangent to $M$ which is not proportional to $\xi$ at the point $p\in M$, 
we define the angle $\theta(X)$ between $\phi X$ and $TM$. Then $M$ is said to be {\it{slant}} \cite{Cab1}, 
if the angle $\theta(X)$ is constant for all $X\in T_pM-\{\xi_p\}$ and $p$ in $M$ i.e., $\theta(X)$ is 
independent of the choice of the vector field $X$ and the point $p\in M$. The angle $\theta(X)$ is 
called the slant angle. Obviously, if $\theta=0$, then $M$ is invariant and if $\theta=\frac{\pi}{2}$, 
then $M$ is anti-invariant. If $M$ is neither invariant nor anti-invariant, then it is proper slant.

Now, we give the following examples of invariant, anti-invariant and slant submanifolds of almost contact metric manifolds.

\begin{example} \rm{Consider $5$-Euclidean space ${\mathbb{R}^5}$ with the cartesian coordinates $(x_1,\,x_2,\,y_1,\,y_2,\,t)$ and the almost contact structure
\begin{align*}
\phi\left(\frac{\partial}{\partial x_i}\right)=-\frac{\partial}{\partial y_i},~~~~\phi\left(\frac{\partial}{\partial y_j}\right)=\frac{\partial}{\partial x_j},~~~~\phi\left(\frac{\partial}{\partial t}\right)=0,~~~~1\leq i, j\leq2.
\end{align*}
It is easy to show that $(\phi, \xi, \eta, g)$ is an almost contact metric structure on ${\mathbb{R}^5}$ with $\xi=\frac{\partial}{\partial t},\,\eta=dt$ and $g$, the Euclidean metric of ${\mathbb{R}^5}$ (for instance, see \cite{U1,U2}). Let $M$ be a submanifold of ${\mathbb{R}^5}$ defined by the immersion $\psi$ as follows
\begin{align*}
\psi(u,\,v,\,t)=(u+v,\,0,\,u-v,\,0,\,t).
\end{align*}
Then, the tangent bundle $TM$ of $M$ is spanned by the following vector fields
\begin{align*}
Z_1=\frac{\partial}{\partial x_1}+\frac{\partial}{\partial y_1},\,\,\,\,Z_2=\frac{\partial}{\partial x_1}-\frac{\partial}{\partial y_1},\,\,\,\,Z_3=\frac{\partial}{\partial t}.
\end{align*}
Clearly, we find
\begin{align*}
\phi Z_1=-\frac{\partial}{\partial y_1}+\frac{\partial}{\partial x_1},\,\,\,\,\phi Z_2=-\frac{\partial}{\partial y_1}-\frac{\partial}{\partial x_1},\,\,\,\,\phi Z_3=0.
\end{align*}
It is easy to see that $M$ is an invariant submanifold of ${\mathbb{R}^5}$ such that $\xi$ is tangent to $M$.}
\end{example}

\begin{example} \rm{Consider a submanifold $M$ of ${\mathbb{R}^7}$ with the cartesian coordinates $(x_1,\,x_2,\,x_3,\,y_1,\,y_2,\,y_3,\,t)$ and the contact structure
\begin{align*}
\phi\left(\frac{\partial}{\partial x_i}\right)=-\frac{\partial}{\partial y_i},\,\,\,\phi\left(\frac{\partial}{\partial y_j}\right)=\frac{\partial}{\partial x_j},\,\,\,\phi\left(\frac{\partial}{\partial t}\right)=0,\,\,\,1\leq i, j\leq3.
\end{align*}
Then it is easy to check that $(\phi, \xi, \eta, g)$ is an almost contact metric structure on ${\mathbb{R}^7}$ with $\xi=\frac{\partial}{\partial t},\,\eta=dt$ and $g$, the Euclidean metric of ${\mathbb{R}^7}$ (see \cite{U1,U2}). Consider an immersion $\chi$ on ${\mathbb{ R}}^7$ defined by
\begin{align*}
\chi(\theta,\,\psi,\,t)=\left(\cos(\theta+\psi),\,\cos(\theta-\psi),\,\theta+\psi,\,\sin(\theta+\psi),\,\sin(\theta-\psi),\,-\theta-\psi,\,t\right).
\end{align*}
Then the tangent space $TM$ of the submanifold $M$ defined by the immersion $\chi$ is spanned by the following vector fields
\begin{align*}
X_1&=-\sin(\theta+\psi)\frac{\partial}{\partial x_1}-\sin(\theta-\psi)\frac{\partial}{\partial x_2}+\frac{\partial}{\partial x_3}+\cos(\theta+\psi)\frac{\partial}{\partial y_1}\\
&+\cos(\theta-\psi)\frac{\partial}{\partial y_2}-\frac{\partial}{\partial y_3},
\end{align*}
\begin{align*}
X_2&=-\sin(\theta+\psi)\frac{\partial}{\partial x_1}+\sin(\theta-\psi)\frac{\partial}{\partial x_2}+\frac{\partial}{\partial x_3}+\cos(\theta+\psi)\frac{\partial}{\partial y_1}\\
&-\cos(\theta-\psi)\frac{\partial}{\partial y_2}-\frac{\partial}{\partial y_3},\,\,\,\,X_3=\frac{\partial}{\partial t}.
\end{align*}
Thus, we obtain
\begin{align*}
\phi X_1&=\sin(\theta+\psi)\frac{\partial}{\partial y_1}+\sin(\theta-\psi)\frac{\partial}{\partial y_2}-\frac{\partial}{\partial y_3}+\cos(\theta+\psi)\frac{\partial}{\partial x_1}\\
&+\cos(\theta-\psi)\frac{\partial}{\partial x_2}-\frac{\partial}{\partial x_3},
\end{align*}
\begin{align*}
\phi X_2&=\sin(\theta+\psi)\frac{\partial}{\partial y_1}-\sin(\theta-\psi)\frac{\partial}{\partial y_2}-\frac{\partial}{\partial y_3}+\cos(\theta+\psi)\frac{\partial}{\partial x_1}\\
&-\cos(\theta-\psi)\frac{\partial}{\partial x_2}-\frac{\partial}{\partial x_3},\,\,\,\,\phi X_3=0.
\end{align*}
It is clear that $\phi X_1$ and $\phi X_2$ are orthogonal to $TM$ and hence $M$ is an anti-invariant submanifold of ${\mathbb{R}^7}$ such that $X_3=\xi$ is tangent to $M$.}
\end{example}

\begin{example} \rm{Let $M$ be a submanifold of ${\mathbb{R}^7}$ with an almost contact structure defined in Example 2.2. Consider the immersion $\psi$ defined as
\begin{align*}
\psi(u,\,v,\,t)=\left(\sin u,\,\sin v,\,u+v,\,\cos u,\,\cos v,\,u-v,\,t\right).
\end{align*}
If, we put
\begin{align*}
U_1=\cos u\frac{\partial}{\partial x_1}+\frac{\partial}{\partial x_3}-\sin u\frac{\partial}{\partial y_1}+\frac{\partial}{\partial y_3},
\end{align*}
\begin{align*}
U_2=\cos v\frac{\partial}{\partial x_2}+\frac{\partial}{\partial x_3}-\sin v\frac{\partial}{\partial y_2}-\frac{\partial}{\partial y_3};\,\,\,\,U_3=\frac{\partial}{\partial t},
\end{align*}
then the restriction of $\{U_1,U_2,U_3\}$ to $M$ forms an orthogonal frame fields of the tangent bundle $TM$. Clearly, we have
\begin{align*}
\phi U_1=-\cos u\frac{\partial}{\partial y_1}-\frac{\partial}{\partial y_3}-\sin u\frac{\partial}{\partial x_1}+\frac{\partial}{\partial x_3},
\end{align*}
\begin{align*}
\phi U_2=-\cos v\frac{\partial}{\partial y_2}-\frac{\partial}{\partial y_3}-\sin v\frac{\partial}{\partial x_2}-\frac{\partial}{\partial x_3};\,\,\,\,\phi U_3=0.
\end{align*}
Then $M$ is a proper slant submanifold of ${\mathbb{R}^7}$ with slant angle $\theta=\cos^{-1}\left(\frac{2}{3}\right)$ such that $U_3=\xi$ is tangent to $M$.}
\end{example}

\begin{example} \rm{Consider a submanifold $M$ of ${\mathbb{R}^5}$ with an almost contact structure defined in Example 2.1. Let us define the immersion of $M$ as follows
\begin{align*}
\chi(u,\,v,\,t)=(u,\,u+v,\,v,\,u-v,\,t).
\end{align*}
Clearly, we have
\begin{align*}
U_1=\frac{\partial}{\partial x_1}+\frac{\partial}{\partial x_2}+\frac{\partial}{\partial y_2},\,\,\,\,U_2=\frac{\partial}{\partial x_2}+\frac{\partial}{\partial y_1}-\frac{\partial}{\partial y_2},\,\,\,\,U_3=\frac{\partial}{\partial t}.
\end{align*}
Then, we get
\begin{align*}
\phi U_1=-\frac{\partial}{\partial y_1}-\frac{\partial}{\partial y_2}+\frac{\partial}{\partial x_2},\,\,\,\,\phi U_2=\frac{\partial}{\partial x_1}-\frac{\partial}{\partial y_2}-\frac{\partial}{\partial x_2},\,\,\,\,\phi U_3=0.
\end{align*}
Thus $M$ is a slant submanifold of ${\mathbb{R}^5}$ with slant angle $\theta=\cos^{-1}\left(\frac{1}{3}\right)$ such that $\xi$ is tangent to $M$.}
\end{example}

We can construct many more examples of such submanifolds. For more examples of slant submanifolds of almost Hermitian manifolds and almost contact metric manifolds we refer to B. Y. Chen' book \cite{Chen2} and \cite{Cab1}.

Next, from the Gauss and Weingarten formulas, we obtain
\begin{align}
\label{eqn2.30}
\overline{R}(X,Y)Z = R(X,Y)Z + A_{h(X,Z)}Y - A_{h(Y,Z)}X,
\end{align}
where ${R}(X,Y)Z$ denotes the tangential part of the curvature tensor of the submanifold.\\
On an invariant submanifold of a generalized Sasakian-space-form $\overline{M}$, we have
\begin{align}
\label{eqn2.31}
h(X,\xi) = 0.
\end{align}
Now, we have
\begin{proposition}
Let $M$ be an invariant submanifold of a generalized Sasakian-space-form $\overline{M}$. Then the following relations hold:
\begin{align}
\label{eqn2.32}
\nabla_{X}\xi = -(f_1-f_3)\phi X,
\end{align}
\begin{align}
\label{eqn2.33}
R(X,Y)\xi=(f_1-f_3)\big[\eta(Y)X-\eta(X)Y\big],
\end{align}
\begin{align}
\label{eqn2.34}
S(X,\xi)=2n(f_1-f_3)\eta(X),
\end{align}
\begin{align}
\label{eqn2.35}
(\nabla_{X}\phi)(Y)=(f_1-f_3)\big[g(X,Y)\xi-\eta(Y)X\big],
\end{align}
\begin{align}
\label{eqn2.36}
h(X,\phi Y)= \phi h(X,Y).
\end{align}
\end{proposition}
\begin{proof} Since $M$ is an invariant submanifold of a generalized Sasakian-space-form $\overline{M}$, then by virtue of \eqref{eqn2.7}, (\ref{eqn2.16}), (\ref{eqn2.31})  we get (\ref{eqn2.32}). Also, from the covariant derivative formula for $\phi$ and (\ref{eqn2.16}), we derive
\begin{align}
\label{eqn2.39}
(\overline{\nabla}_{X}\phi)(Y) = (\nabla_{X}\phi)(Y) + h(X,\phi Y) - \phi h(X,Y).
\end{align}
Comparing the tangential and normal parts of (\ref{eqn2.39}) with (\ref{eqn2.6}), we
get the relations (\ref{eqn2.35}) and (\ref{eqn2.36}). Also, from (\ref{eqn2.30}), we get
\begin{align}
\label{eqn2.40}
\overline{R}(X,Y)\xi = R(X,Y)\xi + A_{h(X,\xi)}Y - A_{h(Y,\xi)}X.
\end{align}
Using \eqref{eqn2.11}, \eqref{eqn2.14} and (\ref{eqn2.31}) in (\ref{eqn2.40}), we get the relation (\ref{eqn2.33}) and (\ref{eqn2.34}).
\end{proof}

Let $\overline{M}^{2n+1}(f_{1},f_{2},f_{3})$ be a generalized Sasakian-space-form and $\overline{\nabla}$ be the Levi-Civita
connection on $\overline{M}$. A linear connection $\widetilde{\overline{\nabla}}$ on $\overline{M}^{2n+1}(f_{1},f_{2},f_{3})$
is said to be {\it{semi-symmetric}} if the torsion tensor $\tau$ of the connection $\widetilde{\overline{\nabla}}$ is given by $$\tau(X,Y)=\widetilde{\overline{\nabla}}_{X}Y-\widetilde{\overline{\nabla}}_{Y}X-[X,Y]$$
satisfies $$\tau(X,Y)=\eta(Y)X-\eta(X)Y$$
for all $X,Y \in \Gamma(TM)$. A semi-symmetric connection $\widetilde{\overline{\nabla}}$ is called {\it{semi-symmetric metric connection}}
if it further satisfies $$\widetilde{\overline{\nabla}}g=0.$$
The relation between the semi-symmetric metric connection $\widetilde{\overline{\nabla}}$ and the Riemannian connection $\overline{\nabla}$
of a generalized Sasakian-space-form $\widetilde{\overline{M}}^{2n+1}(f_{1},f_{2},f_{3})$ is given by \cite{SHAIKH1}
\begin{align}
\label{eqn2.41}
 \widetilde{\overline{\nabla}}_{X}Y= \overline{\nabla}_X Y+\eta(Y)X-g(X,Y)\xi.
\end{align}
 If $\overline{R}$ and $\widetilde{\overline{R}}$ are respectively the Riemannian Curvature tensor of generalized Sasakian-space-form $\overline{M}^{2n+1}(f_{1},f_{2},f_{3})$ with respect to Levi-Civita connection and semi-symmetric metric connection, then we have
\begin{align}
\label{eqn2.42}
\widetilde{\overline{R}}(X,Y)Z &= \overline{R}(X,Y)Z-\alpha(Y,Z)X+\alpha(X,Z)Y\\
\nonumber &+ g(Y,Z)LX +g(X,Z)LY,
\end{align}
where $\alpha$ is a $(0,\,2)$ tensor field given by
\begin{align}
\label{eqn2.43}
  \alpha(X,Y)=(\widetilde{\overline{\nabla}}_{X}\eta)(Y)+\frac{1}{2}g(X,Y),
\end{align}
\begin{align*}
LX=\widetilde{\overline{\nabla}}_{X}\xi+\frac{1}{2}X
\end{align*}
and
\begin{align*}
g(LX,Y)=\alpha(X,Y).
\end{align*}
From (\ref{eqn2.42}), we get
\begin{align}
\label{eqn2.44}
\widetilde{\overline{S}}(X,Y)=\overline{S}(X,Y)-(2n-1)\alpha(X,Y)-ag(X,Y)
\end{align}
and
\begin{align}
\label{eqn2.45}
\widetilde{\overline{r}}=\overline{r}-4na,
\end{align}
where $a=\rm{trace}(\alpha)$, $\widetilde{\overline{S}}$ and $\widetilde{\overline{r}}$ are the Ricci tensor and scalar curvature with respect to
semi-symmetric metric connection $\widetilde{\overline{\nabla}}$ and $\overline{S}$ and $\overline{r}$ are the Ricci tensor and scalar curvature
of $\overline{M}^{2n+1}(f_1,f_2,f_3)$ with respect to Levi-Civita connection, respectively.\\
From (\ref{eqn2.42}) and (\ref{eqn2.44}), we get
\begin{align}
\label{eqn2.46}
  \widetilde{\overline{R}}(X,Y)\xi &= (f_{1}-f_{3})[\eta(Y)X-\eta(X)Y]\\
\nonumber &- \eta(Y)LX+\eta(X)LY,
\end{align}
\begin{align}
\label{eqn2.47}
  \widetilde{\overline{R}}(\xi,X)Y &= (f_{1}-f_{3})[g(X,Y)\xi-\eta(Y)X]\\
  \nonumber &- \alpha(X,Y)\xi+\eta(Y)LX,
\end{align}
\begin{align}
\label{2.48}
  \widetilde{\overline{S}}(X,\xi) = [2n(f_{1}-f_{3})-a]\eta(X)
\end{align}
for arbitrary vectors fields $X,Y$ and $Z$ on $\overline M$.\\
\indent As a generalization of quasi-Einstein (or $\eta$-Einstein) manifolds, recently Shaikh \cite{SHAIKH} introduced
the notion of pseudo quasi-Einstein (or pseudo $\eta$-Einstein) manifolds. A generalized Sasakian-space form is said to be
pseudo quasi-Einstein (or pseudo $\eta$-Einstein) manifold if its Ricci tensor $S$ of the type (0,2) is not identically
zero and satisfies the following:
%%%%%%%%%%%%%%%%%%%%%%%%%%%%%%%%%%%%%%%%%%
\begin{equation}
\label{eqn2.49}
S(X,Y)=pg(X,Y)+q\eta(X)\eta(Y)+sD(X,Y),
\end{equation}
%%%%%%%%%%%%%%%%%%%%%%%%%%%%%%%%%%%%%%%%%%%
where $p,q,s$ are scalars of which $q\neq0$, $s\neq0$ and $D(X,\xi)=0$ for
any vector field $X$. It may be noted that every quasi-Einstein (or $\eta$-Einstein) manifold is a pseudo quasi-Einstein
(or pseudo $\eta$-Einstein) manifold but not conversely as follows by various examples given in \cite{SHAIKH}.
%%%%%%%%%%%%%%%%%%%%%%%%%%%%%%%%%%%%%
\section{invariant submanifolds of generalized sasakian-space-forms}
In this section, we study parallel, semiparallel and $2$-semiparallel invariant submanifolds of generalized Sasakian-space-forms. First, we prove the following:
%%%%%%%%%%%%%%%%%%%%%%%%%%
\begin{theorem}
Let $M$ be an invariant submanifold of a generalized Sasakian-space-form $\overline{M}^{2n+1}(f_{1},f_{2},f_{3})$ such that $f_{1}\neq{f_{3}}$. Then
$M$ is totally geodesic if and only if its second fundamental form is parallel.
\end{theorem}
%%%%%%%%%%%%%%%%%%%%%%%%
\begin{proof} Let $M$ be an invariant submanifold of a generalized Sasakian-space-form $\overline{M}$ such that $f_{1}\neq f_{3}$. Since $h$ is parallel, we have $(\nabla_{X} h)(Y,Z)=0$, which implies
\begin{align}\label{eqn3.1}
\nabla^{\bot}_{X}h(Y,Z)-h(\nabla_{X}Y,Z)-h(Y,\nabla_{X}Z)=0
\end{align}
Putting $Z=\xi$ in (\ref{eqn3.1}) and using (\ref{eqn2.31}), we get
\begin{align}\label{eqn3.2}
h(Y,\nabla_{X}\xi)=0
\end{align}
Using (\ref{eqn2.32}), we arrive at
\begin{align}\label{eqn3.3}
(f_{1}-f_{3})h(Y,\phi X)=0.
\end{align}
Since $f_{1}\neq f_{3}$, thus from above relation we get $h=0$, that is $M$ is totally geodesic.
The converse part is trivial and consequently, we get the desired result.
\end{proof}
\begin{corollary}
An invariant submanifold of a Sasakian-space-form is totally geodesic if and only if its second fundamental form is parallel.
\end{corollary}
\begin{theorem}
Let $M$ be an invariant submanifold of a generalized Sasakian-space-form
$\overline{M}^{2n+1}(f_1,f_2,f_3)$ with $f_1 \neq f_3$ and $r\neq 2n(2n+1)(f_1-f_3)$. Then $M$ is totally geodesic
if and only if $M$ is concircularly semiparallel.
\end{theorem}
\begin{proof} Let $M$ be an invariant submanifold of a generalized Sasakian-space-form
$\overline{M}^{2n+1}(f_1,f_2,f_3)$ such that $r\neq 2n(2n+1)(f_1-f_3)$.
Let $M$ be concircularly semiparallel, i.e., $M$ satisfies the relation $\overline{C}(X,Y)\cdot h = 0$.
Then from (\ref{eqn2.28}) we get
\begin{align}
\label{eqn3.4}
R^\perp(X,Y)h(Z,U)- h\left(C(X,Y)Z,U\right)- h\left(Z,C(X,Y)U\right) = 0.
\end{align}
Setting $X = U = \xi$ in (\ref{eqn3.4}) and using (\ref{eqn2.31}) we obtain
\begin{align}
\label{eqn3.5}
h\left(Z,C(\xi,Y)\xi\right) = 0.
\end{align}
By virtue of (\ref{eqn2.26}) and (\ref{eqn2.31}) it follows from (\ref{eqn3.5}) that
$\big[f_1-f_3-\frac {r}{2n(2n+1)}\big]h(Z,Y) = 0$, which gives $h(Z,Y) = 0$,
since $f_1 \neq f_3$ and $r\neq 2n(2n+1)(f_1-f_3)$ and hence the submanifold $M$ is totally geodesic.
Converse is trivial and hence the proof is complete.
\end{proof}
\begin{corollary}
An invariant submanifold of a Sasakian-space-form is totally geodesic if and only if its second fundamental form
is concircularly semiparallel with $r\neq 2n(2n+1)$.
\end{corollary}
Also we have from Theorem 3.2 that
\begin{corollary}
Let $M$ be an invariant submanifold  of a generalized Sasakian-space-form $\overline{M}^{2n+1}(f_{1},f_{2},f_{3})$ with $f_{1}\neq f_{3}$. Then $M$ is totally geodesic if and only if $M$ is semiparallel.
\end{corollary}
\begin{corollary}
An invariant submanifold of a Sasakian-space-form is totally geodesic if and only if its second fundamental form is semiparallel.
\end{corollary}
\begin{theorem}
Let $M$ be an invariant submanifold of a generalized Sasakian-space-form
$\overline{M}^{2n+1}(f_1,f_2,f_3)$ such that $f_1 \neq f_3$ and $r\neq 2n(2n+1)(f_1-f_3)$. Then $M$ is totally geodesic
if and only if $M$ is concircularly 2-semiparallel.
\end{theorem}
\begin{proof} Let $M$ be an invariant submanifold of a generalized Sasakian-space-form
$\overline{M}^{2n+1}(f_1,f_2,f_3)$ such that $r\neq 2n(2n+1)(f_1-f_3)$.
Let $M$ be concircularly 2-semiparallel. Then $\overline{C}(X,Y)\cdot \overline{\nabla}h = 0$ and hence from (\ref{eqn2.29}) we get
\begin{align}
\label{eqn3.6}
&R^\perp(X,Y)(\overline{\nabla}h)(Z,U,W) - (\overline{\nabla}h)\left(C(X,Y)Z,U,W\right)\\
\nonumber&-(\overline{\nabla}h)\left(Z,C(X,Y)U,W\right)-(\overline{\nabla}h)\left(Z,U,C(X,Y)W\right) = 0.
\end{align}
Putting $X = U = \xi$ in (\ref{eqn3.6}), we obtain
\begin{align}
\label{eqn3.7}
&R^\perp(\xi,Y)(\overline{\nabla}h)(Z,\xi,W) - (\overline{\nabla}h)\left(C(\xi,Y)Z,\xi,W\right)\\
\nonumber&-(\overline{\nabla}h)\left(Z,C(\xi,Y)\xi,W\right)-(\overline{\nabla}h)\left(Z,\xi,C(\xi,Y)W\right) = 0.
\end{align}
By virtue of  (\ref{eqn2.12}), (\ref{eqn2.20}), (\ref{eqn2.26}), (\ref{eqn2.27}), (\ref{eqn2.31}) and (\ref{eqn2.32}), we get
\begin{align}
\label{eqn3.8}
(\overline{\nabla}h)(Z,\xi,W) &= (\overline{\nabla}_{Z}h)(\xi,W)\\
\nonumber&=\nabla_{Z}^\perp\left(h(\xi,W)\right)- h(\nabla_{Z}\xi,W) - h(\xi,\nabla_{Z}W)\\
\nonumber&=(f_1-f_3) h(\phi Z, W),
\end{align}
\begin{align}
\label{eqn3.9}
&(\overline{\nabla}h)(C(\xi,Y)Z,\xi,W) = (\overline{\nabla}_{C(\xi,Y)Z}h)(\xi,W)\\
\nonumber&={\nabla}^\perp_{C(\xi,Y)Z}\left(h(\xi,W)\right)-h(\nabla_{C(\xi,Y)Z}\xi,W)-h(\xi,\nabla_{C(\xi,Y)Z}W)\\
\nonumber&=(f_1-f_3)h\left({\phi}C(\xi,Y)Z,W\right)\\
\nonumber&=-(f_1-f_3)\big[f_1-f_3-\frac{r}{2n(2n+1)}\big]\eta(Z)h(\phi Y,W),
\end{align}
\begin{align}
\label{eqn3.10}
&(\overline{\nabla}h)(Z,C(\xi,Y)\xi,W) = (\overline{\nabla}_Zh\left(C(\xi,Y)\xi,W\right)\\
\nonumber&={\nabla}^{\perp}_Z\left(h(C(\xi,Y)\xi,W)\right)- h\left(\nabla_Z C(\xi,Y)\xi,W\right)- h\left(C(\xi,Y)\xi,\nabla_ZW\right)\\
\nonumber&=-\big[f_1-f_3-\frac{r}{2n(2n+1)}\big]\big[\nabla^\perp_Z h(Y,W)\\
\nonumber&+h\left(\nabla_Z{\eta(Y)\xi-Y},W\right)-h(Y,\nabla_Z W)\big],
\end{align}
\begin{align}
\label{eqn3.11}
&(\overline{\nabla}h)(Z,\xi,C(\xi,Y)W) = (\overline{\nabla}_Z h)\left(\xi,C(\xi,Y)W\right)\\
\nonumber&={\nabla}^{\perp}_Z\left(h(\xi,C(\xi,Y)W)\right)- h\left(\nabla_Z\xi,C(\xi,Y)W\right)- h\left(\xi,\nabla_ZC(\xi,Y)W\right)\\
\nonumber&=(f_1-f_3)h\left(\phi{Z},C(\xi,Y)W\right)\\
\nonumber&=-(f_1-f_3)\big[f_1-f_3-\frac{r}{2n(2n+1)}\big]\eta(W)h(\phi{Z},Y).
\end{align}
In view of (\ref{eqn3.8})-(\ref{eqn3.11}) we have from (\ref{eqn3.7}) that
\begin{align}
\label{eqn3.12}
&(f_1-f_3)R^\perp(\xi,Y)h(\phi{Z},W) + (f_1-f_3)\big[f_1-f_3\\
\nonumber&-\frac{r}{2n(2n+1)}\big]\eta(Z)h(\phi{Y},W)+\big[f_1-f_3-\frac{r}{2n(2n+1)}\big]\big[{\nabla}^{\perp}_Z h(Y,W)\\
\nonumber&+ h\left(\nabla_Z\{\eta(Y)\xi-Y\},W\right)- h(Y,\nabla_Z W)\big]\\
\nonumber&+(f_1-f_3)\big[f_1-f_3-\frac{r}{2n(2n+1)}\big]\eta(W)h(\phi{Z},Y) = 0.
\end{align}
Putting $W = \xi$ in (\ref{eqn3.12}) and using (\ref{eqn2.31}) and (\ref{eqn2.32}) we get
\begin{align*}
2(f_1-f_3)\left(f_1-f_3-\frac{r}{2n(2n+1)}\right)h(Y,\phi Z)= 0,
\end{align*}
which means that either $(1)\,f_1=f_3,$ or $(2)\,\,r=(f_1-f_3)2n(2n+1)$, or $(3)\,\,h=0$. Since neither $f_1=f_3$ nor $r=(f_1-f_3)2n(2n+1)$, then we get $h=0$, that is, $M$ is totally geodesic submanifold of $\overline M$.The converse part of the theorem is obvious and hence the proof is complete.
\end{proof}
\begin{corollary}
An invariant submanifold of a Sasakian-space-form is totally geodesic if and only if
its second fundamental form is concircularly 2-semiparallel with $r \neq 2n(2n+1)$.
\end{corollary}
Also, we have from Theorem 3.3 that
\begin{corollary}
Let $M$ be an invariant submanifold of a generalized Sasakian-space-form
$\overline{M}^{2n+1}(f_{1},f_{2},f_{3})$ with $f_{1}\neq f_{3}$.
Then $M$ is totally geodesic if and only if $M$ is $2$-semiparallel.
\end{corollary}
\begin{corollary}
An invariant submanifold of a Sasakian-space-form is totally geodesic if
and only if its second fundamental form is $2$-semiparallel.
\end{corollary}
\section{Invariant submanifolds with semi-symmetric metric connection}
\indent Let $M$ be an invariant submanifold of a generalized Sasakian-space-form $\overline{M}^{2n+1}(f_{1},f_{2},f_{3})$
with respect to the Levi-Civita connection $\overline{\nabla}$ and
semi-symmetric metric connection $\widetilde{\overline{\nabla}}$.
Let $\nabla$ be the induced connection on $M$ from the connection $\overline{\nabla}$
and $\widetilde{\nabla}$ be the induced connection
on $M$ from the connection $\widetilde{\overline{\nabla}}$.\\
\indent Let $h$ and $\widetilde{h}$ be the second fundamental form with
respect to the Levi-Civita connection and semi-symmetric metric connection,
respectively. Then we have
\begin{align}
\label{eqn4.1}
  \widetilde{\overline{\nabla}}_{X}Y=\widetilde{\nabla}_{X}Y+\widetilde{h}(X,Y).
\end{align}
By virtue of (\ref{eqn2.16}), we get from (\ref{eqn2.41}) and (\ref{eqn4.1}) that
\begin{align}
\label{eqn4.2}\widetilde{\nabla}_{X}Y+\widetilde{h}(X,Y) &= \overline{\nabla}_{X}Y+\eta(Y)X-g(X,Y)\xi \\
\nonumber  &=\nabla_{X}Y+h(X,Y)+\eta(Y)X-g(X,Y)\xi.
\end{align}
Since $M$ is invariant, then by equating the tangential and normal components of (\ref{eqn4.2}), we obtain
\begin{align}
\label{eqn4.3}
\widetilde{\nabla}_{X}Y=\nabla_{X}Y+\eta(Y)X-g(X,Y)\xi
\end{align}
and
\begin{align}
\label{eqn4.4}
\widetilde{h}(X,Y)=h(X,Y).
\end{align}

This leads to the following:
\begin{theorem}
Let $M$ be an invariant submanifold of a generalized Sasakian-space-form $\overline{M}^{2n+1}(f_{1},f_{2},f_{3})$ with the Levi-Civita connection
$\overline{\nabla}$ and semi-symmetric connection $\widetilde{\overline{\nabla}}$ and $\nabla$ be the induced connection on $M$
from the connection $\overline{\nabla}$ and $\widetilde{\nabla}$ be the induced connection on $M$ from the connection
$\widetilde{\overline{\nabla}}$. If $h$ and $\widetilde{h}$ are the second fundamental forms with the Levi-Civita connection
and semi-symmetric metric connection, respectively, then
\begin{enumerate}
\item[(i)] $M$ admits semi-symmetric metric connection.
\item[(ii)] The second fundamental forms with respect to $\nabla$ and $\widetilde{\nabla}$ are equal.
\end{enumerate}
\end{theorem}
From (\ref{eqn2.29a}) and (\ref{eqn4.4}), we get
\begin{align*}
H=\widetilde{H},
\end{align*}
where $H$ and $\widetilde{H}$ are respectively the mean curvature vectors of $M$
with respect to Levi-Civita connection and semi-symmetric metric connection.

This leads to the following:
\begin{theorem}
The mean curvature vectors say $H$ and $\widetilde{H}$ of an invariant submanifold $M$ with respect to the Levi-Civita connection
and semi-symmetric metric connection are same, i.e., $H=\widetilde{H}$.
\end{theorem}
\begin{corollary}
An invariant submanifold $M$ of a generalized Sasakian-space-form $\overline{M}$ endowed with
a semi-symmetric metric connection is minimal with respect to semi-symmetric metric connection
if and only if it is minimal with respect to Levi-Civita connection.
\end{corollary}
\begin{corollary}
$M$ is totally umbilical with respect to semi-symmetric metric connection if and only if it is
totally umbilical with respect to Levi-Civita connection.
\end{corollary}
We can write the equations (\ref{eqn2.20}) and (\ref{eqn2.20a}) with respect to semi-symmetric metric connection as
\begin{align}
\label{eqn4.5}
(\widetilde{\overline{\nabla}}_X h)(Y,Z)=\widetilde{\nabla}_X^{\bot}(h(Y,Z))-h(\widetilde{\nabla}_X Y,Z)-h(Y,\widetilde{\nabla}_X Z),
\end{align}
\begin{align}
\label{eqn4.6}
(\widetilde{\overline{\nabla}}^{2}h)(Z,W,X,Y) &= (\widetilde{\overline{\nabla}}_X {\widetilde{\overline{\nabla}}}_Y h)(Z,W) \\
\nonumber  &= \widetilde{\nabla}_X^{\bot}((\widetilde{\overline{\nabla}}_Y h)(Z,W))-(\widetilde{\overline{\nabla}}_Y h)(\widetilde{\nabla}_X Z,W)\\
\nonumber  &-(\widetilde{\overline{\nabla}}_X h)(Z,\widetilde{\nabla}_Y W)-(\widetilde{\overline{\nabla}}_{{\widetilde{\nabla}}_X Y}h)(Z,W).
\end{align}
Let $M$ be an invariant submanifold of generalized Sasakian-space-form $\overline{M}^{2n+1}(f_1,f_2,f_3)$ with respect to
 semi-symmetric metric connection. Then the second fundamental form $h$ is:
 \begin{enumerate}
\item[(i)] recurrent with respect to semi-symmetric metric connection if
 \begin{align}\label{eqn4.7}
 (\widetilde{\overline{\nabla}}_X h)(Y,Z)=D(X)h(Y,Z),
 \end{align}
 where D is an $1$-form on $M$.
\item[(ii)] $2$-recurrent with respect to semi-symmetric metric connection if
 \begin{align}\label{4.7a}
   (\widetilde{\overline{\nabla}}^{2}h)(Z,W,X,Y)=\psi(X,Y)h(Z,W),
 \end{align}
 where $\psi$ is $2$-form on $M$.
\item[(iii)] generalized $2$-recurrent with respect to semi-symmetric metric connection if
\begin{align}
\label{eqn4.11}
(\widetilde{\overline{\nabla}}^{2}h)(Z,W,X,Y) &= (\widetilde{\overline{\nabla}}_X {\widetilde{\overline{\nabla}}}_Y h)(Z,W) \\
\nonumber  &=\psi(X,Y)h(Z,W)+\rho(X)(\widetilde{\overline{\nabla}}_Y h)(Z,W),
\end{align}
where $\rho(X)$ is an $1$-form on $M$.
\end{enumerate}
We now prove the following:
\begin{theorem}
Let $M$ be an invariant submanifold of a generalized Sasakian-space-form $\overline{M}^{2n+1}(f_1,f_2,f_3)$
with respect to semi-symmetric metric connection such that $(f_1-f_3)^2+1 \neq 0$. Then $h$ is recurrent with respect to semi-symmetric
metric connection if and only if $M$ is totally geodesic.
\end{theorem}
\begin{proof} Let us take the second fundamental form $h$ be recurrent of an invariant submanifold $M$ of a generalized Sasakian-space-form $\overline{M}^{2n+1}(f_1,f_2,f_3)$ with respect to semi-symmetric metric connection. Then we have the relation ({\ref{eqn4.7}}).\\
Putting $Z=\xi$ in (\ref{eqn4.7}) and using (\ref{eqn2.31}) and (\ref{eqn4.5}), we derive
\begin{align}
\label{eqn4.8}
h(Y,\widetilde{\overline{\nabla}}_X\xi)=0.
\end{align}
By virtue of (\ref{eqn2.2}), (\ref{eqn2.7}) and (\ref{eqn2.43}) it follows from (\ref{eqn4.8}) that
\begin{align}
\label{eqn4.9}
   h(Y,X)-(f_1-f_3)h(Y,\phi X)=0.
\end{align}
Interchanging $X$ by $\phi X$ in (\ref{eqn4.9}) and using (\ref{eqn2.1}) and (\ref{eqn2.31}), we find
\begin{align}
\label{eqn4.10}
h(Y,\phi X)+(f_1-f_3)h(Y,X)=0.
\end{align}
From (\ref{eqn4.9}) and (\ref{eqn4.10}) we get $[(f_1-f_3)^2+1]h(X,Y)=0,$ which implies that
either $h(X,Y)=0$ for all $X,Y$ on $M$ or $(f_1-f_3)^2+1=0$, but from the hypothesis of the
theorem $(f_1-f_3)^2+1\neq0$, then from the above relation we get $h=0$ and hence $M$ is totally geodesic.
The converse statement is trivial. This proves the theorem completely.
\end{proof}
%%%%%%%%%%%%%%%%%%
\begin{corollary}
The second fundamental form of an invariant submanifold $M$ of a Sasakian-space-form with respect to semi-symmetric
metric connection is recurrent if and only if $M$ is totally geodesic.
\end{corollary}
\begin{theorem}
Let $M$ be an invariant submanifold of a generalized Sasakian-space-form
$\overline{M}^{2n+1}(f_1,f_2,f_3)$ with respect to semi-symmetric metric connection
such that $(f_1-f_3)^2+1\neq0.$ Then $h$ is generalized $2$-recurrent with
respect to semi-symmetric metric connection if and only if $M$ is totally geodesic.
\end{theorem}
\begin{proof}
Let $h$ be generalized $2$-recurrent with respect to semi-symmetric metric connection. Then we have the relation (\ref{eqn4.11}). Taking $W=\xi$ in (\ref{eqn4.11}) and using (\ref{eqn2.31}), we get
\begin{align}
\label{eqn4.12}
(\widetilde{\overline{\nabla}}_X \widetilde{\overline{\nabla}}_Y h)(Z,\xi)=\rho(X)(\widetilde{\overline{\nabla}}_Y h)(Z,\xi).
\end{align}
Again, using (\ref{eqn2.31}), (\ref{eqn4.5}) and (\ref{eqn4.6}) in (\ref{eqn4.12}), we derive
\begin{align}
\label{eqn4.13}
&-{\widetilde{\nabla}_X} ^\bot(h(Z,\widetilde{\nabla}_Y \xi))+2h(\widetilde{\nabla}_X Z,\widetilde{\nabla}_Y \xi)-\widetilde{\nabla}_X ^\bot h(Z,\widetilde{\nabla}_Y \xi)\\
 \nonumber &+ h(Z,\widetilde{\nabla}_ X \widetilde{\nabla}_Y \xi)+ h(Z,{\widetilde{\nabla}_{\widetilde{\nabla}_X Y}}\xi)+\rho(X)h(Z,\widetilde{\nabla}_Y \xi)=0.
\end{align}
Putting $Z=\xi$ in (\ref{eqn4.13}) and using (\ref{eqn2.31}), we obtain
\begin{align}
\label{eqn4.14}
h(\widetilde{\nabla}_X \xi,\widetilde{\nabla}_Y \xi)=0.
\end{align}
In view of (\ref{eqn2.31}), (\ref{eqn2.32}) and (\ref{eqn4.3}), (\ref{eqn4.14}) yields
\begin{align}
\label{eqn4.15}
{(f_1-f_3)^2} h(\phi X,\phi Y)-(f_1-f_3)[h(\phi X,Y)+h(X,\phi Y)]+h(X,Y) = 0.
\end{align}
Interchanging $X$ by $\phi X$ in (\ref{eqn4.15}) and using (\ref{eqn2.1}) and (\ref{eqn2.31}), we find
\begin{align}
\label{eqn4.16}
-{(f_1-f_3)^2} h(X,\phi Y)+(f_1-f_3)[h(X,Y)-h(\phi X,\phi Y)]+h(\phi X,Y) = 0.
\end{align}
From (\ref{eqn4.15}) and (\ref{eqn4.16}), we get
\begin{align}
\label{eqn4.17}
-(f_1-f_3)h(X,\phi Y)+h(X,Y)=0,
\end{align}
as ${(f_1-f_3)^2}+1 \neq 0.$\\
Interchanging $Y$ by $\phi Y$ in (\ref{eqn4.17}) and using (\ref{eqn2.1}) and (\ref{eqn2.31}), we obtain
\begin{align}
\label{eqn4.18}
(f_1-f_3)h(X,Y)+h(X,\phi Y)=0
\end{align}
From (\ref{eqn4.17}) and (\ref{eqn4.18}), we get $h(X,Y)=0$, as $(f_1-f_3)^2+1\neq 0$, which implies that $M$ is totally geodesic. The
converse statement is trivial. This proves the theorem.
\end{proof}

\begin{corollary}
Let $M$ be an invariant submanifold of a generalized Sasakian-space-form
$\overline{M}^{2n+1}(f_1,f_2,f_3)$ with respect to semi-symmetric metric
connection such that $(f_1-f_3)^2+1\neq0.$ Then $h$ is $2$-recurrent with
respect to semi-symmetric metric connection if and only if $M$ is totally geodesic.
\end{corollary}

\begin{corollary}
The second fundamental form $h$ of an invariant submanifold $M$ of a
Sasakian-space-form is $2$-recurrent with respect to semi-symmetric
metric connection if and only if $M$ is totally geodesic.
\end{corollary}

\section{Ricci solitons on invariant submanifolds}
Let us take $(g,\xi,\lambda)$ be a Ricci soliton on a invariant submanifold $M$ of a generalized Sasakian-space-form  $\overline{M}^{2n+1}(f_1,f_2,f_3)$.
Then we have
\begin{equation}\label{eqn5.1}
  (\pounds _\xi g)(Y,Z)+2S(Y,Z)+2\lambda g(Y,Z)=0.
\end{equation}
From (\ref{eqn2.32}), we get
\begin{eqnarray}
\label{eqn5.2}
(\pounds _\xi g)(Y,Z)&=&g(\nabla_Y\xi,Z)+g(Y,\nabla_Z\xi)\\
\nonumber &=&0.
\end{eqnarray}
Using $(\ref{eqn5.2})$ in $(\ref{eqn5.1})$ we get
\begin{equation*}
  S(Y,Z)=-\lambda g(Y,Z).
\end{equation*}
This leads to the following:
\begin{theorem}
If $(g,\xi,\lambda)$ is a Ricci soliton on an invariant submanifold $M$ of a generalized Sasakian-space-form  $\overline{M}^{2n+1}(f_1,f_2,f_3)$,
then $M$ is Einstein.
\end{theorem}
Now we take $(g,\xi,\lambda)$ is a Ricci soliton on an invariant submanifold $M$ of a
generalized Sasakian-space-form  $\overline{M}^{2n+1}(f_1,f_2,f_3)$
with respect to semi-symmetric metric connection $\widetilde{\overline{\nabla}}$. Then we have
\begin{equation}\label{eqn5.3}
(\widetilde{\pounds}_\xi g)(Y,Z)+2\widetilde{ S}(Y,Z)+2\lambda g(Y,Z)=0.
\end{equation}
From $(\ref{eqn4.3})$, we get
%%%%%%%%%%%%%%%%%%%%%%%%%%%%
\begin{equation}\label{eqn5.4}
\widetilde{\nabla}_X\xi =X-\eta(X)\xi-(f_1-f_3)\phi X.
\end{equation}
%%%%%%%%%%%%%%%%%%%%%%%%%%%
In view of $(\ref{eqn5.4})$, we get
\begin{eqnarray}
\label{eqn5.5}
(\widetilde{\pounds}_\xi g)(Y,Z) &=& g(\widetilde{\nabla}_Y\xi,Z)+g(Y,\widetilde{\nabla}_Z\xi)\\
\nonumber &=& 2[g(Y,Z)-\eta(Y)\eta(Z)].
\end{eqnarray}
Using (\ref{eqn5.3}) we can compute that
\begin{equation}\label{eqn5.6}
\widetilde{S}(Y,Z)=S(Y,Z)-(2n-1)\alpha (X,Y)-ag(X,Y),
\end{equation}
where $\alpha(X,Y)=g(LX,Y)=(\widetilde{\nabla}_X\eta)(Y)+\frac{1}{2}g(X,Y)$
and $a=\text{trace}(\alpha)$.\\
In view of $(\ref{eqn5.5})$ and $(\ref{eqn5.6})$ we get
\begin{equation*}
S(Y,Z) = (a-\lambda-1)g(Y,Z)+\eta(Y)\eta(Z) + (2n-1)\alpha(Y,Z),
\end{equation*}
which implies that $M$ is pseudo $\eta$-Einstein.\\
This leads to the following:
%%%%%%%%%%%%%%%%%%%%%%%%%%%
\begin{theorem}
If $(g,\xi,\lambda)$ is a Ricci soliton on an invariant submanifold $M$ of a generalized Sasakian-space-form  $\overline{M}^{2n+1}(f_1,f_2,f_3)$,
with respect to semi-symmetric metric connection, then $M$ is pseudo $\eta$-Einstein.
\end{theorem}
%%%%%%%%%%%%%%%%%%%%%%%%%%%%%%%%%%%
\section{conclusion}
A Riemannian manifold with constant sectional curvature $c$ is called a real-space-form and its curvature tensor $\overline{R}$
satisfies the condition (\ref{eqn1.1}). Models for these spaces are the Euclidean spaces ($c=0$),
the spheres ($c>0$) and the hyperbolic spaces ($c<0$).

In contact metric geometry, a Sasakian manifold with constant $\phi$-sectional curvature
is called Sasakian-space-form and the curvature tensor of such a manifold is given by (\ref{eqn1.2}).
These spaces can also be modeled depending on $c>-3$, $c=3$ or $c<-3$.

A generalized Sasakian-space-form can be regarded as a generalization of Sasakian-space-form.
By virtue of Theorem $3.1$, Corollary $3.3$, Corollary $3.6$, Theorem $3.2$, and Theorem $3.3$, we can state the following:
\begin{theorem}
Let $M$ be an invariant submanifold of a generalized Sasakian-space-form
$\overline{M}^{2n+1}(f_1,f_2,f_3)$. Then the following statements are equivalent :
\begin{enumerate}
\item[(i)] $M$ is totally geodesic,
\item[(ii)]  $M$ is parallel with $f_1 \neq f_3$,
\item[(iii)] $M$ is semiparallel with $f_1 \neq f_3$,
\item[(iv)] $M$ is $2$-semiparallel with $f_1 \neq f_3$,
\item[(v)] $M$ is concircularly semiparallel with $f_1 \neq f_3$ and $r \neq {2n(2n+1)(f_1-f_3)}$,
\item[(vi)]  $M$ is concircularly 2-semiparallel with $f_1 \neq f_3$ and $r \neq {2n(2n+1)(f_1-f_3)}$.
\end{enumerate}
\end{theorem}

Again from Corollary 3.1, Corollary 3.4, Corollary 3.7, Corollary 3.2 and Corollary 3.5,
we can state the following:
\begin{corollary}
In an invariant submanifold of a Sasakian-space-form, the following statements are equivalent:
\begin{enumerate}
\item[(i)] the submanifold is totally geodesic,
\item[(ii)] the second fundamental form of the submanifold is parallel,
\item[(iii)] the second fundamental form of the submanifold is semiparallel,
\item[(iv)] the second fundamental form of the submanifold is 2-semiparallel,
\item[(v)] the second fundamental form of the submanifold is concircularly semiparallel with $r\neq{2n(2n+1)}$,
\item[(vi)] the second fundamental form of the submanifold is concircularly 2-semiparallel with $r\neq{2n(2n+1)}.$
\end{enumerate}
\end{corollary}

The invariant submanifold $M$ of a generalized Sasakian-space-form $\overline{M}^{2n+1}(f_1,f_2,f_3)$
with respect to semi-symmetric metric connection is also studied. It is shown that $M$ also admits
semi-symmetric metric connection and the second fundamental forms with respect to the Levi-Civita connection and semi-symmetric metric connections are equal.

By virtue of Theorem 4.3, Theorem 4.4 and Corollary 4.4, we can state the following:

\begin{theorem}
Let $M$ be an invariant submanifold of a generalized Sasakian-space-form $\overline{M}^{2n+1}(f_1,f_2,f_3)$ with respect to
semi-symmetric metric connection such that $(f_1-f_3)^2+1 \neq0.$ Then the following statements are equivalent:
\begin{enumerate}
\item[(i)] $M$ is totally geodesic,
\item[(ii)] $h$ is recurrent with respect to semi-symmetric metric connection,
\item[(iii)] $h$ is $2$-recurrent with respect to semi-symmetric metric connection,
\item[(iv)] $h$ is generalized $2$-recurrent with respect to semi-symmetric metric connection.
\end{enumerate}
\end{theorem}
\begin{corollary}
Let $M$ be an invariant submanifold of a Sasakian-space-form with respect to
semi-symmetric metric connection. Then the following statements are equivalent:
\begin{enumerate}
\item[(i)] $M$ is totally geodesic,
\item[(ii)] $h$ is recurrent with respect to semi-symmetric metric connection,
\item[(iii)] $h$ is $2$-recurrent with respect to semi-symmetric metric connection,
\item[(iv)] $h$ is generalized $2$-recurrent with respect to semi-symmetric metric connection.
\end{enumerate}
\end{corollary}
%%%%%%%%%%%%%%%%%%%%%%%%%%%%%%%%
In section 5, we have studied invariant submanifolds of generalized
Sasakian-space-forms $\overline{M}^{2n+1}(f_1,f_2,f_3)$ whose metric are Ricci solitons.
From Theorem 5.1 and Theorem 5.2, we can state the following:\\
%%%%%%%%%%%%%%%%%%%%
\medskip
\noindent{\bf Theorem 6.3.} Let $(g,\xi,\lambda)$ be a Ricci soliton on an invariant submanifold $M$ of a generalized
Sasakian-space-form $\overline{M}^{2n+1}(f_1,f_2,f_3)$. Then the following holds:
\begin{center}
\begin{tabular}{|c|c|}
\hline  connection of $\overline{M}$ & $M$ \\
\hline Riemannian & Einstein\\
\hline semi-symmetric metric & pseudo $\eta$-Einstein\\
\hline
\end{tabular}
\end{center}
%%%%%%%%%%%%%%%%%%%%%%%%%%%%%
\vspace{0.1in}
\noindent{\bf Acknowledgement:} The first author (S. K. Hui)  gratefully acknowledges to
the SERB (Project No.: EMR/2015/002302), Govt. of India for financial assistance of the work. The third author (A.H. Alkhaldi) would like to express his gratitude to King Khalid University, Saudi Arabia for providing administrative and technical support.
%%%%%%%%%%%%%%%%%%%%%%%%%%%%%%%%%%%%%%%%%%%%%%%

%%%%%%%%%%%%%%%%%%%%%%%%%%%%%%%%%%%%%%%%%%%%%%%%%%%%%%%%%%%%%%%%%%%%%%%%%%%%%%
\vspace{0.1in}

Shyamal Kumar Hui

\noindent Department of Mathematics, The University of Burdwan, Golapbag, Burdwan 713104, West Bengal, India

\noindent E-mail: skhui@math.buruniv.ac.in\\

Siraj Uddin

\noindent Department of Mathematics, Faculty of Science, King Abdulaziz University, Jeddah 21589, Saudi Arabia

\noindent E-mail: siraj.ch@gmail.com\\

ALi H. Alkhaldi

\noindent Department of Mathematics, College of Science, King Khalid University, P.O. Box 9004, Abha, Saudi Arabia

\noindent E-mail: ahalkhaldi@kku.edu.sa\\

 Pradip Mandal

\noindent Department of Mathematics, The University of Burdwan, Golapbag, Burdwan 713104, West Bengal, India

\noindent E-mail: pradip2621994@rediffmail.com
\end{document}